\documentclass[a4paper,twoside]{article}

\usepackage[utf8]{inputenc}
\usepackage[T1]{fontenc}
\usepackage{hyperref}
\hypersetup{colorlinks,linkcolor={red!50!black},citecolor={blue!50!black},urlcolor={green!50!black}}
\usepackage{amsfonts,amssymb,bm}
\usepackage{amsmath}
\numberwithin{equation}{section}
\usepackage{tikz-cd}
\usepackage{amsthm}
\newtheoremstyle{mytheoremstyle}{7pt}{7pt}{\normalfont}{}{\normalfont\bfseries}{:}{.5em}{}
\theoremstyle{mytheoremstyle}
\newtheorem{definition}{Definition}[section]
\newtheorem{lemma}[definition]{Lemma}
\newtheorem{proposition}[definition]{Proposition}
\newtheorem{theorem}[definition]{Theorem}
\newtheorem{corollary}[definition]{Corollary}
\newtheorem{example}[definition]{Example}
\newtheorem{remark}[definition]{Remark}

\newcommand{\SetFont}[1]{\mathbb{#1}}
\newcommand{\setN}{\SetFont{N}}
\newcommand{\setZ}{\SetFont{Z}}
\newcommand{\setR}{\SetFont{R}}
\newcommand{\Ker}{\operatorname{Ker}}

\newcommand{\smooth}{\mathcal{C}^\infty}
\newcommand{\anchor}{\mathsf{a}}
\newcommand{\dif}{\mathrm{d}}
\newcommand{\difDeRham}{\bm{\dif}}
\newcommand{\opeLie}{\mathcal{L}}
\newcommand{\opeLieDeRham}{\bm{\opeLie}}
\newcommand{\opeIns}{\iota}
\newcommand{\opeInsDeRham}{\bm{\opeIns}}
\newcommand{\End}{\operatorname{End}}
\newcommand{\Hom}{\operatorname{Hom}}

\begin{document}

\title{\bfseries Twisted cohomology of Lie algebroids}
\author{Benjamin \textsc{Couéraud}\\ LAREMA -- UMR CNRS 6093 -- FR CNRS 2962\\ Université d'Angers -- UBL\\ 2 boulevard Lavoisier, 49045 Angers Cedex 01, France\\ \href{mailto:coueraud@math.univ-angers.fr}{\texttt{coueraud@math.univ-angers.fr}}}
\date{}
\maketitle

\begin{abstract}
	In this short note we define a new cohomology for a Lie algebroid $\mathcal{A}$, that we call the \emph{twisted cohomology} of $\mathcal{A}$ by an odd cocycle $\theta$ in the Lie algebroid cohomology of $\mathcal{A}$. We proof that this cohomology only depends on the Lie algebroid cohomology class $[\theta]$ of the odd cocycle $\theta$. We give a few examples showing that this new cohomology encompasses various well-known cohomology theories.
\end{abstract}

\tableofcontents

\addcontentsline{toc}{section}{Introduction}
\section*{Introduction}

Lie algebroids are objects in differential geometry that generalize both Lie algebras and tangent bundles of smooth manifolds (see the definition \ref{def:Lie-algebroid}). They appear in many situations, ranging from classical geometric classification problems (see \cite{MR2503824} for instance) to Poisson geometry (see \cite{MR2455155}) and foliations (see \cite{MR2012261}), but not exclusively. Importantly, a Lie algebroid is the infinitesimal version of a Lie groupoid (see \cite{MR0216409}). To any Lie algebroid is associated a cohomology algebra, which is a generalization of both the Chevalley-Eilenberg cohomology of a Lie algebra and the De Rham cohomology of a manifold. After reviewing some definitions in section \ref{sec:basics}, we define in section \ref{sec:twisted} a new cohomology, that we will call the \emph{twisted} cohomology of a Lie algebroid, using an odd differential form $\theta$ that is $\dif$-closed, where $\dif$ is the differential of the usual Lie algebroid cohomology. We show in the corollary \ref{co:independence} that this new cohomology only depends on the (Lie algebroid) cohomology class $[\theta]$ of $\theta$, and recover examples of the literature in section \ref{subsec:twisted-examples}, in particular the \emph{twisted} De Rham cohomology of a smooth manifold \cite{MR2838268}.

\section{Lie algebroids}
\label{sec:basics}

In this section, we briefly recall the notion of a Lie algebroid, its cohomology, including the one with \emph{coefficients}, obtained from a representation, and give some examples. All manifolds that we consider are \emph{smooth} in the sense of \cite[chapter 1]{MR2954043}, and all vector bundles are \emph{smooth}, \emph{real} and of \emph{constant rank} (see \cite[definition 1.1, chapter 3]{MR1249482}). The usual interior product, De Rham differential, and Lie derivative, operating on differential forms of a manifold, will be typed in bold type ($\opeInsDeRham_X$, $\difDeRham$ et $\opeLieDeRham_X$, where $X\in\mathfrak{X}(M)$ is a vector field on the manifold $M$).

\subsection{Definition and examples of Lie algebroids}

\begin{definition}
	Let $M$ be a manifold. An \emph{anchor} on a vector bundle $A\to M$ is vector bundle morphism $\anchor:A\to TM$ covering the identity, where $TM\to M$ denotes the tangent bundle of $M$.
\end{definition}

\begin{definition}[{\cite[definition 3.3.1]{MR2157566}}]\label{def:Lie-algebroid}
	Let $M$ be a manifold. A \emph{Lie algebroid} $\mathcal{A}$ is a triple $(A\to M,\anchor,[\cdot,\cdot])$, where $A\to M$ is a vector bundle on $M$, $\anchor$ is an anchor on $A\to M$, and $[\cdot,\cdot]$ is a $\setR$-bilinear operation on the $\smooth(M)$-module $\Gamma(A)$ of sections of $A\to M$ called the \emph{bracket}, such that $(\Gamma(A),[\cdot,\cdot])$ is a Lie algebra and a \emph{right Leibniz rule} is satisfied:
	\begin{equation*}
		[u,fv]=f[u,v]+(\anchor(u)\cdot f)v,
	\end{equation*}
	for all $u,v\in\Gamma(A)$ and $f\in\smooth(M)$.
\end{definition}

\begin{remark}
	From the definition above, a \emph{left} Leibniz rule can be obtained using the skew-symmetry of the bracket $[\cdot,\cdot]$, that is,
	\begin{equation*}
		[fu,v]=f[u,v]-(\anchor(v)\cdot f)u,
	\end{equation*}
	for all $u,v\in\Gamma(A)$ and $f\in\smooth(M)$.
\end{remark}

\begin{proposition}[{\cite[section 6.1]{Kosmann1990}}]
	Let $\mathcal{A}=(A\to M,\anchor,[\cdot,\cdot])$ be a Lie algebroid. The anchor $\anchor$ induces a morphism of Lie algebras $\Gamma(A)\to\mathfrak{X}(M)$, still denoted by $\anchor$, where $\mathfrak{X}(M)$ is equipped with the Lie bracket of vector fields on the manifold $M$.
\end{proposition}

We now give some examples of Lie algebroids.

\begin{example}\label{ex:Lie-algebra}
	Lie algebras are in correspondence with Lie algebroids over a point.
\end{example}

\begin{example}
	Let $M$ be a manifold and $A\to M$ be a Lie algebra bundle. Then $A\to M$ can be equipped with a Lie algebroid structure as follows. Let $(\mathfrak{g},[\cdot,\cdot]_\mathfrak{g})$ be a typical fiber of $A\to M$, that is, a Lie algebra. Then, the Lie algebra bracket $[\cdot,\cdot]_\mathfrak{g}$ induces a bracket on the space of sections $\Gamma(A)$, defined by $[u,v]_x=[u_x,v_x]_\mathfrak{g}$ for all $u$, $v\in\Gamma(A)$ and $x\in M$. This bracket is $\smooth(M)$-bilinear since
	\begin{equation*}
		[fu,v]_x=\big[f(x)u_x,v_x\big]_\mathfrak{g}=f(x)[u_x,v_x]_\mathfrak{g}=\big(f[u,v]\big)_x,
	\end{equation*}
	for all $u$, $v\in\Gamma(A)$, $f\in\smooth(M)$ and $x\in M$. Therefore, $A\to M$ equipped with the null anchor and the bracket we just defined is a Lie algebroid. Lie algebra bundles are studied in detail in \cite[section 3.3]{MR2157566} and \cite{Gundogan}.
\end{example}

\begin{example}\label{ex:canonical}
	Let $M$ be a manifold. On the tangent bundle $TM\to M$ of $M$, there is a natural Lie algebroid structure given as follows. The anchor is the identity vector bundle morphism $TM\to TM$, and the bracket on $\Gamma(TM)=\mathfrak{X}(M)$ is the Lie bracket of vector fields on $M$ (see \cite[chapter 8]{MR2954043}). This Lie algebroid will be called the \emph{canonical} Lie algebroid on $M$, and will be denoted by $\mathcal{T}_M$.
\end{example}

\begin{example}\label{ex:foliation}
	Let $M$ be a manifold and $A\to M$ be an involutive distribution of $TM\to M$ (see \cite[chapter 19]{MR2954043}). Note that according to the global Frobenius theorem \cite[theorem 19.21]{MR2954043}, this distribution comes from a foliation of $M$. The vector bundle $A\to M$ is equipped with a Lie algebroid structure as follows. The anchor is the inclusion map $A\hookrightarrow TM$, and the bracket is the Lie bracket of vector fields on $M$ restricted to $\Gamma(A)$, which is well-defined thanks to the involutivity of the distribution.
\end{example}

\begin{example}\label{ex:Poisson}
	Let $M$ be a manifold and $\pi\in\Gamma(\Lambda^2 TM)$ a be bivector field such that $[\pi,\pi]_\textsf{SN}=0$ for the Schouten-Nijenhuis bracket \cite[section 3.3]{MR2906391}. Then $(M,\pi)$ is a Poisson manifold \cite[section 1.3.2]{MR2906391} and the cotangent bundle $T^*M\to M$ of $M$ can be given a Lie algebroid structure as follows (see \cite[section 6.2]{MR2455155}). The anchor $\anchor$ is defined by
	\begin{equation*}
		\anchor(\alpha)\cdot f=\pi(\alpha,\difDeRham f),
	\end{equation*}
	for all $\alpha\in\Omega^1(M)=\Gamma(T^*M)$ and $f\in\smooth(M)$. Writing $\pi^\sharp:\Omega^1(M)\to\mathfrak{X}(M)$ for the map defined by $\pi^\sharp(\alpha)(\beta)=\pi(\alpha,\beta)$, we get $\anchor=\pi^\sharp$. The bracket is defined by
	\begin{equation*}
		[\alpha,\beta]=\opeLieDeRham_{{\pi^\sharp}(\alpha)}\beta-\opeLieDeRham_{{\pi^\sharp}(\beta)}\alpha-\difDeRham\big(\pi(\alpha,\beta)\big),
	\end{equation*}
	for all $\alpha$, $\beta\in\Omega^1(M)$. Then the Jacobi identity for this bracket is equivalent to the identity $[\pi,\pi]_\textsf{SN}=0$, where $[\cdot,\cdot]_{\textsf{SN}}$ denotes the Schouten-Nijenhuis bracket (see \cite[section 3.3]{MR2906391}). We will denote by $\mathcal{P}_M[\pi]$ this Lie algebroid.
\end{example}

\begin{example}\label{ex:action}
	Let $M$ be a manifold and let $(\mathfrak{g},[\cdot,\cdot]_\mathfrak{g})$ be a Lie algebra acting (infinitesimally) on $M$ through a Lie algebra homomorphism $\rho:\mathfrak{g}\to\mathfrak{X}(M)$. The trivial vector bundle $M\times\mathfrak{g}$ can be equipped with a Lie algebroid structure (called an \emph{action Lie algebroid}, see \cite[example 3.3.7]{MR2157566}). The anchor is defined by $\anchor((x,\xi))=\rho(\xi)_x$, for any $x\in M$ and $\xi\in\mathfrak{g}$. The bracket is defined on sections $u$, $v\in\Gamma(M\times\mathfrak{g})\cong\smooth(M,\mathfrak{g})$ by
	\begin{equation*}
		[u,v]_x=[u_x,v_x]_{\mathfrak{g}}+\big[\rho(u_x)\cdot v\big]_x-\big[\rho(v_x)\cdot u\big]_x,
	\end{equation*}
	where the dot $\cdot$ in the right hand side denotes the component-wise action of a vector field.
\end{example}

Some other important examples have been omitted, but we still would like to briefly describe them hereafter. Similarly to Lie algebras that can be obtained by differentiation from a Lie group, starting with a Lie groupoid \cite[section 1.1]{MR2157566}, one can obtain a Lie algebroid by differentiation \cite[section 3.5]{MR2157566}. Nevertheless, there are obstructions to integrate a given Lie algebroid into a Lie groupoid \cite{Crainic2003}. The last example is actually the Lie algebroid of a Lie groupoid, the so-called \emph{action Lie groupoid} \cite[example 1.1.9]{MR2157566}, which gives the corresponding action Lie algebroid in the case the infinitesimal action integrates into a global one (see \cite[example 3.5.14]{MR2157566}). Finally, although we will not need it, we mention the so-called \emph{Atiyah Lie algebroid} (\cite{MR0086359}, \cite[section 3.2]{MR2157566}). Let $G$ be a Lie group and $\pi:P\to M$ be a principal $G$-bundle. Let $\overline{\difDeRham\pi}$ denote the quotient map $TP/G\to TM$. Taking $\overline{\difDeRham\pi}$ as anchor and the Lie bracket of ($G$-invariant) vector fields on the manifold $P$ as bracket, the vector bundle $TP/G\to M$ is a Lie algebroid, called the Atiyah Lie algebroid associated to $\pi:P\to M$. It is obtained by differentiation from the \emph{Ehresmann gauge groupoid} associated to the same principal $G$-bundle (see \cite{MR1129261} and the references within).

\subsection{Cohomology of Lie algebroids}

We now turn to the definition of Lie algebroid cohomology.

\begin{definition}
	Let $\mathcal{A}=(A\to M,\anchor,[\cdot,\cdot])$ be a Lie algebroid. We will denote by $\Omega^\bullet(\mathcal{A})$ the $\setZ$-graded commutative $\setR$-algebra $(\Gamma(\Lambda^\bullet A^*),\wedge)$. Elements of $\Omega^\bullet(\mathcal{A})$ will be called \emph{differential forms on $\mathcal{A}$}.
\end{definition}

\begin{definition}[{\cite[section 2]{MR1726784}, \cite[section 5.2]{MR2455155}}]\label{def:exterior-derivative}
	Let $\mathcal{A}=(A\to M,\anchor,[\cdot,\cdot])$ be a Lie algebroid. The \emph{exterior derivative} of $\mathcal{A}$ is the operator $\dif:\Omega^p(\mathcal{A})\to\Omega^{p+1}(\mathcal{A})$ defined by
	\begin{multline*}
		\dif\alpha(u_0,\dots,u_p)=\sum_{i=0}^p(-1)^i\anchor(u_i)\cdot\alpha(u_0,\dots,\widehat{u_i},\dots,u_p)\\
		+\sum_{0\leq i<j\leq p}(-1)^{i+j}\alpha\big([u_i,u_j],u_0,\dots,\widehat{u_i},\dots,\widehat{u_j},\dots,u_p\big),
	\end{multline*}
	for all $u_0,\dots,u_p\in\Gamma(A)$ and any $\alpha\in\Omega^p(\mathcal{A})$; where $\anchor(u_i)\cdot\alpha(u_0,\dots,\widehat{u_i},\dots,u_p)$ denotes the action of the vector field $\anchor(u_i)$ on the function $\alpha(u_0,\dots,\widehat{u_i},\dots,u_p)$. On functions $f\in\smooth(M)$, $\dif$ is defined by $(\dif f)(u)=\anchor(u)\cdot f$, for any $u\in\Gamma(A)$.
\end{definition}

The following proposition is a particular case of \cite[proposition 5.2.3, point 2]{MR2455155}.

\begin{proposition}[{\cite[proposition 5.2.3]{MR2455155}}]\label{pr:exterior-derivative-derivation}
	Let $\mathcal{A}=(A\to M,\anchor,[\cdot,\cdot])$ be a Lie algebroid. The exterior derivative $\dif$ of $\mathcal{A}$ is a derivation of degree 1 of $\Omega^\bullet(\mathcal{A})$ that squares to zero, that is, $\dif$ is a differential and $\Omega^\bullet(\mathcal{A})$ endowed with this differential becomes a differential graded commutative algebra (see \cite[part 1, chapter 3]{MR1802847} for the definition).
\end{proposition}

\begin{definition}
	Let $\mathcal{A}=(A\to M,\anchor,[\cdot,\cdot])$ be a Lie algebroid. The \emph{cohomology} of $\mathcal{A}$ is defined as the homology of the differential graded commutative algebra $(\Omega^\bullet(\mathcal{A}),\dif)$ (see \cite[part 1, chapter 3]{MR1802847} for the definition).
\end{definition}

\begin{lemma}
	Let $\mathcal{A}=(A\to M,\anchor,[\cdot,\cdot])$ be a Lie algebroid. On $H^\bullet(\mathcal{A})$ there is a \emph{cup product} that gives $H^\bullet(\mathcal{A})$ the structure of a graded commutative algebra.
\end{lemma}

\begin{proof}
	The cup product comes directly from the wedge product on $\Omega^\bullet(\mathcal{A})$, it is defined on cohomology classes by $[\omega]\wedge[\eta]=[\omega\wedge\eta]$, for any $\omega\in\Omega^p(\mathcal{A})$ and $\eta\in\Omega^q(\mathcal{A})$ that are cocycles, that is, $\dif$-closed. The differential form $\omega\wedge\eta$ is $\dif$-closed since $\dif$ is a derivation according to proposition \ref{pr:exterior-derivative-derivation}. Also, the cup product does not depend on the choice of representation for the cocycles since, according to proposition \ref{pr:exterior-derivative-derivation} again, we have
	\begin{equation*}
		(\omega+\dif\alpha)\wedge(\eta+\dif\beta)=\omega\wedge\eta+\dif\big(\alpha\wedge\eta+(-1)^p\omega\wedge\beta+\alpha\wedge\dif\beta\big),
	\end{equation*}	 
	for any $\alpha\in\Omega^\bullet(\mathcal{A})$ and $\beta\in\Omega^\bullet(\mathcal{A})$.
\end{proof}

A detailed study of the cohomology of Lie algebroids, and especially the relationship between the cohomology of a Lie groupoid and the cohomology of its associated Lie algebroid, is available in \cite[chapter 7]{MR2157566}. We only explore below some simple examples.

\begin{example}
	Let $\mathcal{A}$ be the Lie algebroid associated to a Lie algebra $\mathfrak{g}$ (see example \ref{ex:Lie-algebra}). Then $H^\bullet(\mathcal{A})$ corresponds to the Chevalley-Eilenberg cohomology of $\mathfrak{g}$ (see \cite[section 14]{Chevalley1948}).
\end{example}

\begin{example}
	Let $M$ be a manifold and consider the canonical Lie algebroid $\mathcal{T}_M$ associated to $M$ (see example \ref{ex:canonical}). The cohomology $H^\bullet(\mathcal{T}_M)$ is the De Rham cohomology of $M$ (see \cite[chapter 17]{MR2954043}).
\end{example}

\begin{example}
	Let $M$ be a manifold and $A\to M$ be an involutive distribution, and consider the associated Lie algebroid $\mathcal{A}$ (see example \ref{ex:foliation}). Then $H^\bullet(\mathcal{A})$ corresponds to the tangential cohomology of the foliation associated to the involutive distribution (see \cite[chapter 3]{Moore1988} and \cite[section 1.1.3]{MR2319199}).
\end{example}

\begin{example}\label{ex:Lichnerowicz-Poisson-cohomology}
	Let $(M,\pi)$ be a Poisson manifold, and consider the associated Lie algebroid $\mathcal{P}_M[\pi]$. Then $H^\bullet(\mathcal{P}_M[\pi])$ corresponds to the Lichnerowicz-Poisson cohomology of $(M,\pi)$ (see \cite{MR0501133} and \cite[definition 5.1]{MR1269545}). Note that the differential can be identified with $-[\pi,\cdot]_\textsf{SN}$ (\cite[proposition 4.3]{MR1269545} and \cite[proposition 6.4]{Kosmann1990}).
\end{example}

\subsection{Representations of Lie algebroids}

The subject of this section is representations of Lie algebroids, which allow to consider \emph{coefficients} in Lie algebroid cohomology.

\begin{definition}
	Let $\mathcal{A}=(A\to M,\anchor,[\cdot,\cdot])$ be a Lie algebroid and $V\to M$ be a vector bundle. We will denote by $\Omega^\bullet(\mathcal{A},V)$ the $\setZ$-graded $\setR$-vector space $\Gamma(\Lambda^\bullet A^*\otimes V)$. Elements of $\Omega^\bullet(\mathcal{A},V)$ will be called \emph{differential forms on $\mathcal{A}$ with values in $V\to M$}. In degree 0, we recover sections of $V\to M$.
\end{definition}

The following result is immediate.

\begin{proposition}\label{pr:graded-module-structure}
	Let $\mathcal{A}$ be a Lie algebroid on a manifold $M$, and $V\to M$ be a vector bundle. $\Omega^\bullet(\mathcal{A},V)$ is a graded $\Omega^\bullet(\mathcal{A})$-module for the multiplication defined by $\alpha\wedge(\beta\otimes s)=(\alpha\wedge\beta)\otimes s$, for any $\alpha$, $\beta\in\Omega^\bullet(\mathcal{A})$ and $s\in\Gamma(V)$.
\end{proposition}

\begin{definition}[{\cite[definition 8.4.7]{MR2795151}, \cite[section 0]{MR1929305}, \cite[section 2]{MR1726784}}]
	Let $\mathcal{A}=(A\to M,\anchor,[\cdot,\cdot])$ be a Lie algebroid, and let $V\to M$ be a vector bundle. A \emph{$\mathcal{A}$-connection} on $V\to M$ is a $\setR$-linear map $\nabla:\Gamma(V)\to\Gamma(A^*\otimes V)$, satisfying the relations
	\begin{gather*}
		\nabla_u(fs)=f\nabla_u s+(\anchor(u)\cdot f)s,\\
		\nabla_{fu}s=f\nabla_u s,
	\end{gather*}
	for any function $f\in\smooth(M)$ and sections $u\in\Gamma(A)$, $s\in\Gamma(V)$, where $\nabla_u$ denotes the map $\Gamma(V)\to\Gamma(V)$, $s\mapsto(\nabla s)(u)$, for any $u\in\Gamma(A)$.
\end{definition}

Similarly to the covariant exterior derivative of a manifold equipped with a (linear) connection (see \cite[theorem 12.57]{MR2572292}), there is a similar operator for Lie algebroids equipped with a representation, whose existence is guaranteed by the following proposition.

\begin{proposition}[{\cite[section 5.2]{MR2455155}}]
	Let $\mathcal{A}=(A\to M,\anchor,[\cdot,\cdot])$ be a Lie algebroid, and let $\nabla$ be a $\mathcal{A}$-connection on a vector bundle $V\to M$. There exists a unique operator $\dif_\nabla:\Omega^p(\mathcal{A},V)\to\Omega^{p+1}(\mathcal{A},V)$ called the \emph{covariant exterior derivative} of $\mathcal{A}$, defined by $\dif_\nabla s=\nabla s$ for all $s\in\Gamma(V)$, and
	\begin{equation*}
		\dif_\nabla(\alpha\otimes s)=\dif\alpha\otimes s+(-1)^p\alpha\wedge\nabla s,
	\end{equation*}
	for any $\alpha\in\Omega^p(\mathcal{A})$ and $s\in\Gamma(V)$; it is then extended to the whole $\Omega^\bullet(\mathcal{A},V)$ by the rules
	\begin{gather}
		\dif_\nabla(\alpha\wedge\omega)=\dif\alpha\wedge\omega+(-1)^p\alpha\wedge\dif_\nabla\omega,
		\label{eq:cov-ext-der-left}\\
		\dif_\nabla(\omega\wedge\alpha)=\dif_\nabla\omega\wedge\alpha+(-1)^q\omega\wedge\dif\alpha,
		\label{eq:cov-ext-der-right}
	\end{gather}
	for all $\alpha\in\Omega^p(\mathcal{A})$ and $\omega\in\Omega^q(\mathcal{A},V)$.
\end{proposition}

\begin{remark}[{\cite[definition 7.1.1]{MR1262213}}]
	The operator $\dif_\nabla$ introduced in the previous proposition admits an intrinsic formula given by
	\begin{multline*}
		\dif_\nabla\omega(u_0,\dots,u_p)=\sum_{i=0}^p(-1)^i\nabla_{u_i}\omega(u_0,\dots,\widehat{u_i},\dots,u_p)\\
		+\sum_{0\leq i<j\leq p}(-1)^{i+j}\omega\big([u_i,u_j],u_0,\dots,\widehat{u_i},\dots,\widehat{u_j},\dots,u_p\big),
	\end{multline*}
	for all $\omega\in\Omega^p(\mathcal{A},V)$ and all $u_0,\dots,u_p\in\Gamma(A)$.
\end{remark}

The following result comes from the definition of a Lie algebroid connection.

\begin{lemma}\label{le:curvature-trilinear}
	Let $\mathcal{A}=(A\to M,\anchor,[\cdot,\cdot])$ be a Lie algebroid, and let $\nabla$ be a $\mathcal{A}$-connection on a vector bundle $V\to M$. The map $\Gamma(A)\times\Gamma(A)\times\Gamma(V)\to\Gamma(V)$, $(u,v,s)\mapsto(\nabla_u\circ\nabla_v-\nabla_v\circ\nabla_u-\nabla_{[u,v]})(s)$ is $\smooth(M)$-trilinear. Thus there exists a $2$-differential form on $\mathcal{A}$ with values in $\End V$, $F_\nabla\in\Omega^2(\mathcal{A},\End V)$, such that $F_\nabla(u,v)=\nabla_u\circ\nabla_v-\nabla_v\circ\nabla_u-\nabla_{[u,v]}$, for all $u$, $v\in\Gamma(A)$. 
\end{lemma}

\begin{definition}
	Let $\mathcal{A}$ be a Lie algebroid on a manifold $M$, and let $\nabla$ be a $\mathcal{A}$-connection on a vector bundle $V\to M$. The $2$-differential form $F_\nabla$ introduced in the previous lemma will be called the \emph{curvature} of the $\mathcal{A}$-connection $\nabla$, and we will say that $\nabla$ is \emph{flat} if $F_\nabla=0$.
\end{definition}

The following lemma is immediate.

\begin{lemma}
	Let $\mathcal{A}$ be a Lie algebroid on a manifold $M$, and let $V\to M$ be a vector bundle. $\Omega^\bullet(\mathcal{A},\End V)$ is a $\setZ$-graded $\setR$-algebra for the multiplication defined by
	\begin{equation*}
		(\alpha\otimes\Phi)\wedge(\beta\otimes\Psi)=(\alpha\wedge\beta)\otimes(\Phi\circ\Psi),
	\end{equation*}
	for any $\alpha$, $\beta\in\Omega^\bullet(\mathcal{A})$ and $\Phi$, $\Psi\in\Gamma(\End V)$. Moreover, $\Omega^\bullet(\mathcal{A},V)$ is a left $\Omega^\bullet(\mathcal{A},\End V)$-module for the multiplication defined by
	\begin{equation*}
		(\alpha\otimes\Phi)\wedge(\beta\otimes s)=(\alpha\wedge\beta)\otimes\Phi(s),
	\end{equation*}
	for any $\alpha$, $\beta\in\Omega^\bullet(\mathcal{A})$, $\Phi\in\Gamma(\End V)$ and $s\in\Gamma(V)$.
\end{lemma}

\begin{proposition}
	Let $\mathcal{A}=(A\to M,\anchor,[\cdot,\cdot])$ be a Lie algebroid, and let $\nabla$ be a $\mathcal{A}$-connection on a vector bundle $V\to M$. We have the formula
	\begin{equation*}
		\dif_\nabla^2\alpha=F_\nabla\wedge\alpha,
	\end{equation*}
	for all $\alpha\in\Omega^\bullet(\mathcal{A},V)$. Therefore, if $\nabla$ is flat, $\dif_\nabla$ is a differential on $\Omega^\bullet(\mathcal{A},V)$.
\end{proposition}

\begin{proof}
	To begin with, we have for any $\omega\in\Omega^1(\mathcal{A},V)$ that
	\begin{equation*}
		(\dif_\nabla\omega)(u,v)=\nabla_u\omega(v)-\nabla_v\omega(u)-\omega([u,v]),
	\end{equation*}
	for all $u$, $v\in\Gamma(A)$. Then, let $s\in\Gamma(V)$, the above formula for $\omega=\nabla s\in\Omega^1(\mathcal{A},V)$ yields
	\begin{equation*}
		\big(\dif_\nabla^2 s\big)(u,v)=\nabla_u\nabla_vs-\nabla_v\nabla_us-\nabla_{[u,v]}s=F_\nabla(u,v)s=(F_\nabla\wedge s)(u,v),
	\end{equation*}
	for all $u$, $v\in\Gamma(A)$.
	Now we prove the formula for any simple tensor  $\alpha\otimes s\in\Omega^p(\mathcal{A},V)$. According to \eqref{eq:cov-ext-der-left}, we have 
	\begin{align*}
		\dif_\nabla^2(\alpha\otimes s)&=\dif_\nabla^2(\alpha\wedge s)\\
		&=\dif_\nabla\big(\dif\alpha\wedge s+(-1)^{p}\alpha\wedge\dif_\nabla s\big)\\
		&=\dif^2\alpha\wedge s+(-1)^{p+1}\dif\alpha\wedge\dif_\nabla s+(-1)^{p}\dif\alpha\wedge\dif_\nabla s+\alpha\wedge\dif_\nabla^2 s\\
		&=\alpha\wedge(F_\nabla\wedge s)\\
		&=F_\nabla\wedge(\alpha\otimes s),
	\end{align*}
	where in the last step we used the fact that $F_\nabla$ is a $2$-differential form on $\mathcal{A}$. The general formula holds for any element of $\Omega^\bullet(\mathcal{A},V)$ using $\setR$-linear combinations of simple tensors.
\end{proof}

\begin{definition}[{\cite[section 8.4]{MR2795151}, \cite[section 1]{MR1726784}}]\label{def:representation}
	Let $\mathcal{A}$ be a Lie algebroid on a manifold $M$. A \emph{representation} of $\mathcal{A}$ (or a \emph{$\mathcal{A}$-module}) is a vector bundle $V\to M$ equipped with a flat $\mathcal{A}$-connection $\nabla$. Such a representation will be denoted by $(V\to M,\nabla)$.
\end{definition}

\begin{definition}
	Let $\mathcal{A}$ be a Lie algebroid on a manifold $M$ and $(V\to M,\nabla)$ be a $\mathcal{A}$-module. We define $H^\bullet(\mathcal{A};V,\nabla)$ as the homology of the differential graded vector space $(\Omega^\bullet(\mathcal{A},V),\dif_\nabla)$. We will call $H^\bullet(\mathcal{A};V,\nabla)$ the cohomology of $\mathcal{A}$ with \emph{coefficients} in $(V\to M,\nabla)$.
\end{definition}

\begin{lemma}\label{le:cohomology-graded-module-structure}
	Let $\mathcal{A}$ be a Lie algebroid on a manifold $M$ and $(V\to M,\nabla)$ be a $\mathcal{A}$-module. Then $H^\bullet(\mathcal{A};V,\nabla)$ is a graded $H^\bullet(\mathcal{A})$-module.
\end{lemma}

\begin{proof}
	The multiplication is directly induced by the one defined in the proposition \ref{pr:graded-module-structure}: $[\alpha]\wedge[\omega]=[\alpha\wedge\omega]$, for any $\alpha\in\Omega^p(\mathcal{A})$ and $\omega\in\Omega^q(\mathcal{A},V)$. According to \eqref{eq:cov-ext-der-left}, it is clear that $\alpha\wedge\omega$ is again $\dif_\nabla$-closed. Also, the multiplication is well-defined, since, using \eqref{eq:cov-ext-der-left} again, we have
	\begin{equation*}
		(\alpha+\dif\beta)\wedge(\omega+\dif_\nabla\eta)=\alpha\wedge\omega+\dif_\nabla\big(\beta\wedge\omega+(-1)^p\alpha\wedge\eta+\beta\wedge\dif_\nabla\eta\big),
	\end{equation*}
	for any $\beta\in\Omega^\bullet(\mathcal{A})$ and $\eta\in\Omega^\bullet(\mathcal{A},V)$.
\end{proof}

\begin{example}\label{ex:trivial-representation}
	Let $\mathcal{A}=(A\to M,\anchor,[\cdot,\cdot])$ be a Lie algebroid. The \emph{trivial representation} of $\mathcal{A}$ is given by the trivial vector bundle $V=M\times\setR\to M$ of rank 1, together with the connection defined by $\nabla_u\lambda=\anchor(u)\cdot\lambda$, for any $u\in\Gamma(A)$ and $\lambda\in\Gamma(V)\cong\smooth(M)$. In this case, $H^\bullet(\mathcal{A};V,\nabla)\cong H^\bullet(\mathcal{A})$.
\end{example}

\begin{example}
	Let $\mathcal{A}$ be a Lie algebroid over a point $M$, that is, a Lie algebra $\mathfrak{g}$ (see example \ref{ex:Lie-algebra}), that we will assume to be of finite dimension. Let $(V\to M,\nabla)$ be a $\mathcal{A}$-module. Since $M$ is a point, $V\to M$ is actually a vector space that we will denote again by $V$. Then, the $\mathcal{A}$-connection $\nabla$ is a linear map $V\to\mathfrak{g}^*\otimes V$. Using the linear isomorphisms
	\begin{equation*}
		\Hom(V,\mathfrak{g}^*\otimes V)\cong\Hom(V,\Hom(\mathfrak{g},V))\cong\Hom(\mathfrak{g}\otimes V,V)\cong\Hom(\mathfrak{g},\End(V)),
	\end{equation*}
	we deduce that, to consider a linear map $V\to\mathfrak{g}^*\otimes V$ is equivalent to consider a linear map $\rho:\mathfrak{g}\to V$. Moreover, this map $\rho$ is also a Lie algebra homomorphism thanks to proposition \ref{le:curvature-trilinear}. Therefore, in this case, a $\mathcal{A}$-module $(V\to M,\nabla)$ corresponds to a $\mathfrak{g}$-module $(V,\rho)$. The covariant exterior derivative is the Chevalley-Eilenberg differential on $\mathfrak{g}$ associated to the $\mathfrak{g}$-module $V$ (see \cite[section 23]{Chevalley1948}) and $H^\bullet(\mathcal{A};V,\nabla)$ is the Chevalley-Eilenberg cohomology $H_\textsf{CE}^\bullet(\mathfrak{g},V)$ of $\mathfrak{g}$ with coefficients in the $\mathfrak{g}$-module $V$.
\end{example}

\begin{example}
	Let $M$ be a manifold and let $\mathcal{T}_M$ denote the canonical Lie algebroid associated to $M$ (see example \ref{ex:canonical}). A representation of $\mathcal{T}_M$ is a vector bundle $V\to M$ together with a flat linear connection $\nabla$ on $V\to M$. The covariant exterior derivative is the differential $\difDeRham_\nabla$ associated to $\nabla$ (see \cite[theorem 12.57]{MR2572292} and \cite[chapitre 7, section 4]{MR0336651}), and $H^\bullet(\mathcal{T}_M;V,\nabla)$ is the homology of the chain complex $(\Gamma(\Lambda^\bullet T^*M\otimes V),\difDeRham_\nabla)$.
\end{example}

\begin{example}
	Let $(M,\pi)$ be a Poisson manifold and let $\mathcal{P}_M[\pi]$ be the associated Lie algebroid (see example \ref{ex:Poisson}). A $\mathcal{P}_M[\pi]$-connection on a vector bundle $V\to M$ corresponds to the notion of a \emph{contravariant connection} on $(M,\pi)$ (see \cite[proposition 2.1.2]{MR1818181}). However, it seems that there is no mention of the associated cohomology in the literature.
\end{example}

\begin{example}
	Let $\mathcal{A}$ be the action Lie algebroid of example \ref{ex:action}. We have a representation of $\mathcal{A}$ given by the trivial vector bundle $M\times\setR\to\setR$ of rank 1, together with the connection $\nabla:\smooth(M)\to\smooth(M,\mathfrak{g}^*)\otimes\smooth(M)$ defined by $\nabla_\xi\lambda=\rho(\xi)\cdot\lambda$, for any $\xi\in\mathfrak{g}$ and $\lambda\in\smooth(M)$. Note that this connection is flat because $\rho:\mathfrak{g}\to\Gamma({TM})$ is a Lie algebra homomorphism. We have $\Omega^\bullet(\mathcal{A},M\times\setR)\cong\smooth(M,\Lambda^\bullet\mathfrak{g}^*)$, and we also have an isomorphism of chain complexes
	\begin{equation*}
		\begin{tikzcd}[row sep=large]
			\Lambda^k\mathfrak{g}^*\otimes\smooth(M)\arrow{d}[swap]{\Phi^k}\arrow{r}{\dif_\mathsf{CE}} &
			\Lambda^{k+1}\mathfrak{g}^*\otimes\smooth(M)\arrow{d}{\Phi^{k+1}}\\
			\smooth(M,\Lambda^k\mathfrak{g}^*)\arrow{r}{\dif} & \smooth(M,\Lambda^{k+1}\mathfrak{g}^*)
		\end{tikzcd},
	\end{equation*}
	where $\dif_\textsf{CE}$ denotes the Chevalley-Eilenberg differential \cite[section 23]{Chevalley1948}, and where $\Phi^k$ denotes the natural map $\Lambda^k\mathfrak{g}^*\otimes\smooth(M)\to\smooth(M,\Lambda^k\mathfrak{g}^*)$, $\omega\otimes f\mapsto f\omega$. This map induces an isomorphism of graded modules between the cohomology $H^\bullet(\mathcal{A};M\times\setR,\nabla)$ on one side, and the Chevalley-Eilenberg cohomology \linebreak $H_\textsf{CE}^\bullet(\mathfrak{g},\smooth(M))$ of $\mathfrak{g}$ with coefficients in the $\mathfrak{g}$-module $(\smooth(M),\rho)$ on the other side.
\end{example}

\section{Twisted cohomology of Lie algebroids}
\label{sec:twisted}

In this section, given a Lie algebroid $\mathcal{A}$ and an odd cocycle $\theta$ with respect to the cohomology of $\mathcal{A}$, we define a new cohomology, that we will call the \emph{twisted cohomology} of $\mathcal{A}$ by $\theta$. We show that it only depends on the cohomology class $[\theta]$ of $\theta$ (in the Lie algebroid cohomology of $\mathcal{A}$). Then we explain how this new cohomology encompasses several cohomology theories that appear in the literature.

\subsection{Definition of the twisted cohomology}

\begin{definition}
	Let $\mathcal{A}$ be a Lie algebroid on a manifold $M$. Define
	\begin{equation*}
		\Omega^{\text{even}}(\mathcal{A})=\bigoplus_{i\geq 0}\Omega^{2i}(\mathcal{A}),\quad\Omega^{\text{odd}}(\mathcal{A})=\bigoplus_{i\geq 0}\Omega^{2i+1}(\mathcal{A}).
	\end{equation*}
	The elements of $\Omega^\text{even}(\mathcal{A})$ are called \emph{even differential forms} on $\mathcal{A}$ and elements of $\Omega^\text{odd}(\mathcal{A})$ are called \emph{odd differential forms} on $\mathcal{A}$.
\end{definition}

This decomposition $\Omega^\bullet(\mathcal{A})=\Omega^{\text{even}}(\mathcal{A})\oplus\Omega^{\text{odd}}(\mathcal{A})$ of differential forms on $\mathcal{A}$ equips $\Omega^\bullet(\mathcal{A})$ with a $\setZ/2\setZ$-graduation, induced by the $\setZ$-graduation. We extend this $\setZ/2\setZ$-graduation to $\Omega^\bullet(\mathcal{A},V)$ as soon as we are provided with a representation $(V\to M,\nabla)$ of $\mathcal{A}$, $M$ being the base of $\mathcal{A}$. 

\begin{definition}
	Let $\mathcal{A}$ be a Lie algebroid on a manifold $M$, $(V\to M,\nabla)$ be a representation of $\mathcal{A}$, and $\theta\in\Omega^\text{odd}(\mathcal{A})$. The \emph{exterior covariant derivative} of $\mathcal{A}$, relatively to $(V\to M,\nabla)$, \emph{twisted} by the odd differential form $\theta$, is the operator $\dif_\nabla[\theta]$ defined by \begin{equation*}{\dif_\nabla[\theta]}\,\omega=\dif_\nabla\omega+\theta\wedge\omega,\end{equation*} for any $\omega\in\Omega^\bullet(\mathcal{A},V)$.
\end{definition}

This operator appeared in \cite[section 3]{GrabowskiMarmo} and \cite[section 3.1]{MR1862087}. Note that, since $\theta$ is odd, the operator ${\dif_\nabla[\theta]}$ is odd, namely ${\dif_\nabla[\theta]}$ maps $\Omega^\text{even}(\mathcal{A},V)$ into $\Omega^\text{odd}(\mathcal{A},V)$ and $\Omega^\text{odd}(\mathcal{A},V)$ into $\Omega^\text{even}(\mathcal{A},V)$.

\begin{lemma}
	Let $\mathcal{A}$ be a Lie algebroid on a manifold $M$, $(V\to M,\nabla)$ be a representation of $\mathcal{A}$, and $\theta\in\Omega^\text{odd}(\mathcal{A})$ be a $\dif$-closed odd differential form on $\mathcal{A}$, where $\dif$ denotes the differential of $\mathcal{A}$ (see definition \ref{def:exterior-derivative}). Under these assumptions, the operator ${\dif_\nabla[\theta]}$ is a differential on the $\setZ/2\setZ$-graded $\setR$-vector space $\Omega^\bullet(\mathcal{A},V)$.
\end{lemma}

\begin{proof}
	For any $\omega\in\Omega^\bullet(\mathcal{A},V)$ we have
	\begin{align*}
		{\dif_\nabla[\theta]}^2\;\omega=\dif_\nabla^2\omega+\dif_\nabla(\theta\wedge\omega)+\theta\wedge\dif_\nabla\omega+\theta\wedge\theta\wedge\omega.
	\end{align*}
	But $\dif_\nabla^2\omega=0$ since $\nabla$ is a flat connection (see the definition \ref{def:representation}). Also $\theta\wedge\theta=0$ since $\theta$ is odd. Indeed, write $\sum_{i\geq 0}\theta_{2i+1}$ for $\theta$, with $\theta_{2i+1}\in\Omega^{2i+1}(\mathcal{A})$. Then $\theta_{2i+1}\wedge\theta_{2i+1}=0$ for any $i\geq 0$ and $\theta_{2i+1}\wedge\theta_{2j+1}=-\theta_{2j+1}\wedge\theta_{2i+1}$ for all $i$, $j\geq0$ such that $i\neq j$. This implies that $\theta\wedge\theta=0$. Using the same decomposition for $\theta$ again, the above formula yields 
	\begin{equation*}
		{\dif_\nabla[\theta]}^2\omega=\sum_{i\geq 0}\left(\dif_\nabla\theta_{2i+1}\right)\wedge\omega,
	\end{equation*}
	which equals zero since $\dif_\nabla\theta=\dif\theta=0$, $\theta$ being assumed to be $\dif$-closed.
\end{proof}

\begin{definition}
	Let $\mathcal{A}$ be a Lie algebroid on a manifold $M$, $(V\to M,\nabla)$ be a representation of $\mathcal{A}$, and $\theta\in\Omega^\text{odd}(\mathcal{A})$ a $\dif$-closed odd differential form on $\mathcal{A}$. Let $H^\bullet(\mathcal{A};V,\nabla;\theta)$ be the homology of the differential graded vector space $(\Omega^\bullet(\mathcal{A},V),{\dif_\nabla[\theta]})$. We will call $H^\bullet(\mathcal{A};V,\nabla;\theta)$ the cohomology of $\mathcal{A}$ relatively to the representation $(V\to M,\nabla)$, \emph{twisted} by the cocycle $\theta$.
\end{definition}

\begin{lemma}
	Let $\mathcal{A}$ be a Lie algebroid on a manifold $M$, and $(V\to M,\nabla)$ be a representation of $\mathcal{A}$, and $\theta\in\Omega^\text{odd}(\mathcal{A})$ a $\dif$-closed odd differential form on $\mathcal{A}$. Then $H^\bullet(\mathcal{A};V,\nabla;\theta)$ is a graded $H^\bullet(\mathcal{A})$-module.
\end{lemma}

\begin{proof}
	The multiplication is the same as in lemma \ref{le:cohomology-graded-module-structure}, only the differential changes. Let $\alpha\in\Omega^p(\mathcal{A})$ be a $\dif$-closed differential form on $\mathcal{A}$ and let $\omega\in\Omega^q(\mathcal{A},V)$ be a ${\dif_\nabla[\theta]}$-closed differential form on $\mathcal{A}$ with values in $V\to M$. Thanks to the identity \eqref{eq:cov-ext-der-left}, we compute that
	\begin{align*}
		{\dif_\nabla[\theta]}(\alpha\wedge\omega)&=\dif_\nabla(\alpha\wedge\omega)+\theta\wedge\alpha\wedge\omega\\
		&=\dif\alpha\wedge\omega+(-1)^p\alpha\wedge\dif_\nabla\omega+\theta\wedge\alpha\wedge\omega\\
		&=\dif\alpha\wedge\omega+(-1)^p\alpha\wedge{\dif_\nabla[\theta]}\omega\\
		&=0,
	\end{align*}
	so the multiplication makes sense, that is we obtain a cocycle. But we also need to show that the multiplication is well-defined, that is it doesn't depend on the choice of representative. Thanks to the identity \eqref{eq:cov-ext-der-left} again, we compute for all $\beta\in\Omega^{p-1}(\mathcal{A})$ and $\eta\in\Omega^{q-1}(\mathcal{A},V)$ that
	\begin{align*}
		{\dif_\nabla[\theta]}\big(\beta\wedge\omega&+(-1)^p\alpha\wedge\eta+\beta\wedge{\dif_\nabla[\theta]}\eta\big)\\
		&=\dif\beta\wedge\omega+(-1)^{p-1}\beta\wedge\dif_\nabla\omega+\theta\wedge\beta\wedge\omega\\
		&\quad\quad+\alpha\wedge\dif_\nabla\eta+(-1)^p\theta\wedge\alpha\wedge\eta\\
		&\quad\quad+\dif\beta\wedge\dif_\nabla[\theta]\eta+(-1)^{p-1}\beta\wedge\dif_\nabla\big(\dif_\nabla[\theta]\eta\big)+\theta\wedge\beta\wedge\dif_\nabla[\theta]\eta\\
		&=\dif\beta\wedge\omega+(-1)^{p-1}\beta\wedge\dif_\nabla[\theta]\omega+\alpha\wedge\dif_\nabla[\theta]\eta\\
		&\quad\quad+\dif\beta\wedge\dif_\nabla[\theta]\eta-(-1)^{p-1}\beta\wedge\theta\wedge\dif_\nabla\eta\\
		&\quad\quad+\theta\wedge\beta\wedge\dif_\nabla\eta+\theta\wedge\beta\wedge\theta\wedge\eta\\
		&=\dif\beta\wedge\omega+\alpha\wedge\dif_\nabla[\theta]\eta+\dif\beta\wedge\dif_\nabla[\theta]\eta.
	\end{align*}
	Therefore we obtain
	\begin{equation*}
		(\alpha+\dif\beta)\wedge\big(\omega+{\dif_\nabla[\theta]}\eta\big)=\alpha\wedge\omega+{\dif_\nabla[\theta]}\big(\beta\wedge\omega+(-1)^p\alpha\wedge\eta+\beta\wedge{\dif_\nabla[\theta]}\eta\big),
	\end{equation*}
	which yields the result.
\end{proof}

\subsection{Independence from the cocycle cohomology class}

We can now state the main theorem of this note.

\begin{theorem}
	Let $\mathcal{A}$ be a Lie algebroid on a manifold $M$, $(V\to M,\nabla)$ be a representation of $\mathcal{A}$, $\theta\in\Omega^\text{odd}(\mathcal{A})$ a $\dif$-closed odd differential form on $\mathcal{A}$ and $\Psi\in\Omega^\text{even}(\mathcal{A})$. We have an isomorphism of graded $H^\bullet(\mathcal{A})$-modules
	\begin{equation*}
		H^\bullet(\mathcal{A};V,\nabla;\theta+\dif\Psi)\cong H^\bullet(\mathcal{A};V,\nabla;\theta).
	\end{equation*}
\end{theorem}

\begin{proof}
	The idea of the proof is the following. Consider $\omega\in\Ker\dif_{\nabla}[\theta+\dif\Psi]$, so $\omega$ satisfies the \emph{differential equation} ${\dif_\nabla[\theta]}\omega+\dif\Psi\wedge\omega=0$. By analogy with calculus, all solutions can be written as $\omega=\exp(-\Psi)\wedge\eta$, where $\eta$ is a \emph{constant} for ${\dif_\nabla[\theta]}$, that is $\eta\in\Ker{\dif_\nabla[\theta]}$, provided that we can define $\exp(-\Psi)$. Thus, the isomorphism should be $\omega\mapsto\exp(\Psi)\wedge\omega$, for any $\omega\in\Omega^\bullet(\mathcal{A},V)$. Now we formulate a rigorous proof. Define $\exp(\Psi)|_x=\sum_{k\geq 0}\frac{\Psi^{\wedge k}|_x}{k!}$, for any $x\in M$; this series converge because exterior powers are known to vanish when the rank of the vector bundle is exceeded. For all real numbers $s$ and $t$ we have that $\exp(s\Psi)\wedge\exp(t\Psi)=\exp((s+t)\Psi)$ thanks to Newton's binomial theorem; in particular for $s=1$ and $t=-1$ we get that the map $\omega\mapsto\exp(-\Psi)\wedge\omega$ is the inverse of the map $\omega\mapsto\exp(\Psi)\wedge\omega$. Moreover, $\Psi$ being even, the map $\omega\mapsto\exp(\Psi)\wedge\omega$ is even, that is, maps $\Omega^\text{even}(\mathcal{A},V)$ into $\Omega^\text{even}(\mathcal{A},V)$ and $\Omega^\text{odd}(\mathcal{A},V)$ into $\Omega^\text{odd}(\mathcal{A},V)$. By induction we have $\dif\Psi^{\wedge k}=k\Psi^{\wedge(k-1)}\wedge\dif\Psi$ for any $k\in\setN^*$, hence $\dif\exp(\Psi)=\exp(\Psi)\wedge\dif\Psi$. For any $\omega\in\Omega^\bullet(\mathcal{A},V)$, we compute that
	\begin{align*}
		{\dif_\nabla[\theta]}\big(\exp(\Psi)\wedge\omega\big)&=\dif_{\nabla}\big(\exp(\Psi)\wedge\omega\big)+\theta\wedge\exp(\Psi)\wedge\omega\\
		&=\dif\exp(\Psi)\wedge\omega+\exp(\Psi)\wedge\dif_\nabla\omega+\exp(\Psi)\wedge\theta\wedge\omega\\
		&=\exp(\Psi)\wedge\dif\Psi\wedge\omega+\exp(\Psi)\wedge\dif_{\nabla}\omega+\exp(\Psi)\wedge\theta\wedge\omega\\
		&=\exp(\Psi)\wedge\dif_{\nabla}[\theta+\dif\Psi](\omega),
	\end{align*}
	which has for consequence that $\varepsilon=\exp(\Psi)\wedge\cdot$ induces a $\setR$-linear isomorphism of $\setZ/2\setZ$-graded vector spaces in cohomology. Furthermore, it is also an isomorphism of graded $H^\bullet(\mathcal{A})$-modules, since we have  $\varepsilon([\alpha]\wedge[\omega])=[\alpha]\wedge\varepsilon([\omega])$ for any $[\alpha]\in H^p(\mathcal{A})$ and $[\omega]\in H^q(\mathcal{A};V,\nabla;\theta)$. Indeed, we compute that
	\begin{equation*}
		\varepsilon\big((\alpha+\dif\beta)\wedge(\omega+\dif_\nabla[\theta]\eta)\big)=\alpha\wedge\varepsilon(\omega)+{\dif_\nabla[\theta+\dif\Psi]}\Gamma,
	\end{equation*}
	where $\Gamma$ is defined by
	\begin{equation*}
		\Gamma=\exp(\Psi)\wedge\dif\beta\wedge\big(\omega+\dif_\nabla[\theta]\eta\big)+(-1)^p\exp(\Psi)\wedge\alpha\wedge\eta,
	\end{equation*}
	for any $\beta\in\Omega^{p-1}(\mathcal{A})$ and $\eta\in\Omega^{q-1}(\mathcal{A},V)$.
\end{proof}

The following corollary is directly deduced from the previous proposition.

\begin{corollary}\label{co:independence}
	Let $\mathcal{A}$ be a Lie algebroid on a manifold $M$, $(V\to M,\nabla)$ be a representation of $\mathcal{A}$, and $\theta\in\Omega^\text{odd}(\mathcal{A})$ be a $\dif$-closed odd differential form on $\mathcal{A}$. Then $H^\bullet(\mathcal{A};V,\nabla;\theta)$ depends only on the (Lie algebroid) cohomology class $[\theta]\in H^\text{odd}(\mathcal{A})=\bigoplus_{i\geq 0}H^{2i+1}(\mathcal{A})$ of $\theta$ and will be denoted by $H^\bullet(\mathcal{A};V,\nabla;[\theta])$.
\end{corollary}

\subsection{Examples of twisted cohomologies}
\label{subsec:twisted-examples}

In this last section, we show that the twisted cohomology of a Lie algebroid $\mathcal{A}$ by an odd cocycle can been found in various places of the literature.

\begin{example}\label{ex:twisted-De-Rham}
	Let $M$ be a manifold and let $H^\bullet_{\textsf{DR}}(M)$ denote the De Rham cohomology of $M$. We are going to apply the corollary \ref{co:independence} to the Lie algebroid $\mathcal{T}_M$ (see example \ref{ex:canonical}). Let $\nabla$ be a flat linear connection on a vector bundle $V\to M$ and $\theta\in\Omega^\text{odd}(M)=\Omega^\text{odd}(\mathcal{T}_M)$ be a $\difDeRham$-closed odd differential form. Then the graded $H^\bullet_{\textsf{DR}}(M)$-module $H^\bullet(\mathcal{T}_M;V,\nabla;\theta)$ only depends on the De Rham cohomology class $[\theta]\in H^\text{odd}_{\textsf{DR}}(M)=H^\text{odd}(\mathcal{T}_M)$. This fact has been proven in \cite[section 1]{MR2838268} were the authors used twisted De Rham complexes to define a notion of twisted analytical torsion. In the case of the trivial representation (see example \ref{ex:trivial-representation}), the differential $\difDeRham+\theta\wedge\cdot$, where $\theta$ is a $\difDeRham$-closed differential form on $M$ of odd degree, has been used in various contexts like in twisted K-theory (see \cite[section 6]{MR2307274} and \cite[section2.1]{MR2365650} for a more topological point of view), the theory of hypergeometric functions \cite[section 2]{adolphson1997}, or in the study of hyperplane arrangements \cite{MR2159008}, and in generalized complex geometry \cite[section 2.1]{MR2811595}, where it can be seen as the natural differential to consider on an exact Courant algebroid with \v{S}evera class a $\difDeRham$-closed $3$-differential form $\theta$ on the base manifold $M$. 
\end{example}

\begin{example}
	Let $\mathcal{A}=(A\to M,\anchor,[\cdot,\cdot])$ be a Lie algebroid and let $[\theta]\in H^1(\mathcal{A})$. We can modify the trivial representation of $\mathcal{A}$ (see \ref{ex:trivial-representation}) by means of the cocycle $\theta$ as follows. We keep the trivial vector bundle $V=M\times\setR\to\setR$ of rank 1 but we now take for connection $\nabla_u\lambda=\anchor(u)\cdot\lambda+\lambda\theta(u)$. Then the differential ${\dif_\nabla[\theta]}$ plays a role in the definition of generalized Lie bialgebroids (see \cite[definition 2.4]{MR2047554} and \cite[section 3.1]{MR1862087}) and also in the study of Dirac structures of such generalized Lie bialgebroids \cite[section 4]{MR2047554}.
\end{example}

\begin{example}
	Let $\mathfrak{g}$ be a Lie algebra, $(V,\rho)$ a $\mathfrak{g}$-module and $\theta$ an element in $\Lambda^\bullet\mathfrak{g}^*$ of odd degree, $\dif_{\textsf{CE}}$-closed. Then the $H_\textsf{CE}(\mathfrak{g})$-module $H^\bullet(\mathfrak{g};V,\rho;\theta)$ only depends on the Chevalley-Eilenberg cohomology class $[\theta]\in H_\textsf{CE}(\mathfrak{g})$. This situation arises when one applies the example \ref{ex:twisted-De-Rham} to a compact simply connected Lie group (see \cite{Chevalley1948}).
\end{example}

\begin{example}
	Let $(M,\pi)$ be a Poisson manifold and consider the associated Lie algebroid $\mathcal{P}_M[\pi]$ (see example \ref{ex:Poisson}) and let $\dif_\pi$ denote its differential (see example \ref{ex:Lichnerowicz-Poisson-cohomology}). Let $R\in\Gamma(\Lambda^3 TM)$ be a $\dif_\pi$-closed $3$-vector field. Then the twisted differential $\dif_\pi+R\wedge\cdot$ is used in \cite[section 4.3]{AMSW} in relation with a Courant algebroid based on Poisson geometry. The associated twisted cohomology for the trivial representation of $\mathcal{P}_M[\pi]$ only depends on the Lichnerowicz-Poisson cohomology class $[R]\in H_\textsf{LP}^3(M,\pi)$.
\end{example}

\end{document}